\documentclass[draft,reqno]{amsart}

\usepackage[Symbol]{upgreek}

\usepackage{amssymb}
\usepackage{epic}
\usepackage{mathrsfs}
\usepackage{accents}
\usepackage{stmaryrd}

\newcommand{\bbold}{\mathbb}
\newcommand{\cal}{\mathcal}

\def \cM{\mathbf{M}}

\def\R { {\bbold R} }

\def\Z { {\bbold Z} }

\def\cC{\cal{C}}
\def\N { {\bbold N} }
\def\T { {\bbold T} }

\def \Def{\operatorname{Def}}

\def \iso{\operatorname{iso}}

\def \f{\operatorname{f}}
\def \trop{\operatorname{trop}}
\def \Th{\operatorname{Th}}

\renewcommand\epsilon{\varepsilon}

\def \<{\langle}
\def \>{\rangle}

\def \((  {(\!(}
\def \)) {)\!)}

\def \k {{{\boldsymbol{k}}}}

\DeclareMathSymbol{\precequ}{\mathrel}{symbols}{"16}
\DeclareMathSymbol{\succequ}{\mathrel}{symbols}{"17}

\newtheorem{theorem}{Theorem}[section]
\newtheorem{lemma}[theorem]{Lemma}
\newtheorem{prop}[theorem]{Proposition}
\newtheorem{cor}[theorem]{Corollary}

\theoremstyle{definition}

\theoremstyle{remark}

\let\oldi\i
\let\oldj\j
\renewcommand\i{\relax\ifmmode{\boldsymbol{i}}\else\oldi\fi}
\renewcommand\j{\relax\ifmmode{\boldsymbol{j}}\else\oldj\fi}

\renewcommand\leq{\leqslant}
\renewcommand\geq{\geqslant}
\renewcommand\preceq{\preccurlyeq}

\renewcommand\le{\leq}
\renewcommand\ge{\geq}

\DeclareMathAlphabet{\mathbf}{OML}{cmm}{b}{it}

\DeclareFontFamily{U}{fsy}{}
\DeclareFontShape{U}{fsy}{m}{n}{<->s*[.9]psyr}{}
\DeclareSymbolFont{der@m}{U}{fsy}{m}{n}
\DeclareMathSymbol{\der}{\mathord}{der@m}{182}

\DeclareSymbolFont{der@m}{U}{fsy}{m}{n}
\DeclareMathSymbol{\derdelta}{\mathord}{der@m}{100}

\DeclareSymbolFont{imag@m}{OT1}{cmr}{m}{ui}
\DeclareMathSymbol{\imag}{\mathord}{imag@m}{105}

\DeclareFontFamily{OMS}{smallo}{}
\DeclareFontShape{OMS}{smallo}{m}{n}{<->s*[.65]cmsy10}{}
\DeclareSymbolFont{smallo@m}{OMS}{smallo}{m}{n}
\DeclareMathSymbol{\smallo}{\mathord}{smallo@m}{79}

\DeclareFontFamily{OMS}{largerdot}{}
\DeclareFontShape{OMS}{largerdot}{m}{n}{<->s*[.8]cmsy10}{}
\DeclareSymbolFont{largerdot@m}{OMS}{largerdot}{m}{n}
\DeclareMathSymbol{\largerdot}{\mathord}{largerdot@m}{15}

\DeclareMathSymbol{\llambda}{\mathord}{der@m}{108}
\DeclareMathSymbol{\rrho}{\mathord}{der@m}{114}

\def \upo{\upomega}

\newcommand{\equationqed}[1]{\[\pushQED{\qed}#1 \qedhere\popQED\]\let\qed\relax}
\newcommand{\alignqed}[1]{\begin{align*}\pushQED{\qed} #1 \qedhere\popQED\end{align*}\let\qed\relax}

\makeatletter
\newcommand{\dminus}{\mathbin{\text{\@dminus}}}

\newcommand{\@dminus}{%
  \ooalign{\hidewidth\raise1ex\hbox{\bf.}\hidewidth\cr$\m@th-$\cr}%
}
\makeatother

\begin{document}

\title{Extending Fubini measures}

%\title{The Dimension of a Definable Set of Transseries}

%\author[Aschenbrenner]{Matthias Aschenbrenner}
%\address{Department of Mathematics\\
%University of California, Los Angeles\\
%Los Angeles, CA 90095\\
%U.S.A.}
%\email{matthias@math.ucla.edu}

\author[van den Dries]{Lou van den Dries}
\address{Department of Mathematics\\
University of Illinois at Urbana-Cham\-paign\\
Urbana, IL 61801\\
U.S.A.}
\email{vddries@illinois.edu}

%\author[van der Hoeven]{Joris van der Hoeven}
%\address{\'Ecole Polytechnique\\
%91128 Palaiseau Cedex\\
%France}
%\email{vdhoeven@lix.polytechnique.fr}

\begin{abstract} Let $C\subseteq M$ be stably embedded in a structure $\cM=(M;\dots)$. We consider {\em Fubini measures} on the subcategory $\Def(C)$ of the category $\Def(\cM)$ of definable sets in $\cM$, with ``Fubini" signaling good behaviour in definable families. We show that such a Fubini measure extends uniquely to the larger subcategory of $\Def(\cM)$ whose objects are the sets that are ``fiberable over $C$". In cases of interest ``fiberable over $C$"  coincides with ``co-analyzable relative to $C$." This applies in particular  to the differential field $\T$ of transseries with $C=\R$, and to differentially closed fields with constant field $C$. \end{abstract}

\date{December 2025}

\maketitle

\section{Introduction}\label{int}

\noindent
Semialgebraic sets $Y\subseteq \R^n$ have a dimension $\dim(Y)\in \N\cup\{-\infty\}$ and an Euler characteristic $E(Y)\in \Z$ with good properties; see \cite[Ch. 4]{D}. In \cite{ADH} we asked for a natural extension of this dimension and Euler characteristic to the definable sets $X\subseteq \T^m$
that are ``fiberable by $\R$" as defined there; here $\T$ is the real closed differential field of transseries, which has its subfield $\R$ as constant field. 

There are many similar situations, and so it makes sense to do the above as a special case of a general extension result for what we call {\em Fubini measures}, with ``Fubini'' signaling good behavior in families. For example, at the end of Section~\ref{special} we shall also apply this to differentially closed fields (models of $\operatorname{DCF}_0$).   

\medskip\noindent
Let $\cM=(M;\cdots)$ be a one-sorted $\cal{L}$-structure;
``definable" will mean ``definable in $\cM$ with parameters" unless we indicate otherwise.  For definable $X\subseteq M^m$ and an elementary extension $\cM^*=(M^*;\cdots)$ of $\cM$ we let $X^*\subseteq (M^*)^m$ be defined in $\cM^*$ by any $\cal{L}_M$-formula that defines $X$ in $\cM$. 
The (small) category $\Def(\cM)$ has as its objects the definable sets $X\subseteq M^m$; the number $m\in \N$ is part of  specifying $X$ as an object of $\Def(\cM)$.  For objects $X,Y$ of $\Def(\cM)$, the morphisms $X\to Y$ are the definable maps $X\to Y$,
and composition of morphisms is composition of maps, so an isomorphism $X\to Y$ is the same as a definable bijection $X\to Y$. We shall abuse language by
indicating the class of objects of a category $\cal{C}$ also by $\cal{C}$, so $X\in \cal{C}$ means that $X$ is an object of $\cal{C}$. 

\medskip\noindent
Let $\cC$ be a full subcategory of $\Def(\cM)$ such that for all $X,Y\in \cC$: \begin{itemize}
\item[(a)] if $X, Y\subseteq M^m$, then $X\cup Y\subseteq M^m$ is an object of $\cC$; 
\item[(b)] every definable subset of $X$ is an object of $\cC$;
\item[(c)]  $X\times Y\in \cC$ (where $X\times Y\subseteq  M^{m+n}$ for $X\subseteq M^m, Y\subseteq M^n$).
\end{itemize} 
Let $A$ be a commutative semiring, that is, a set with binary operations of addition and multiplication and distinguished elements
$0$ and $1$ such that $A$ is a commutative monoid with respect to both addition and multiplication, with $0$ and $1$ as neutral elements for addition and multiplication, respectively, such that multiplication is distributive with respect to addition, and such that also $0\cdot a=0$ for all $a\in A$.

Let  $\mu$ an $A$-valued Fubini measure on $\cC$, that is, $\mu: \cC\to A$, and for all $X,Y\in  \cC$:\begin{enumerate} 
\item[(d)] $\mu(X)=0$ whenever $X$ is empty; $\mu(X)=1$ whenever $X$ is a singleton; 
\item[(e)] $\mu(X\cup Y)=\mu(X)+\mu(Y)$ whenever $X,Y\subseteq M^m$ are disjoint;
\item[(f)] every $\cC$-morphism $f: X \to Y$ satisfies ``Fubini": \begin{enumerate}
\item[(F1)] the subset $\{\mu\big(f^{-1}(y)\big):\ y\in Y\}$ of $A$ is finite;
\item[(F2)] for all $a\in A$, the set 
$Y_a := \{y\in Y:\ \mu\big(f^{-1}(y)\big)=a\}$ is definable;
\item[(F3)] if $\mu\big(f^{-1}(y)\big)$ takes the constant value $a$ on $Y$, then
$\mu(X)= a\mu(Y)$.
\end{enumerate} 
\end{enumerate} 
Thus for $X,Y\in  \cC$:  $\mu(X\times Y)=\mu(X)\cdot \mu(Y)$; if $X$ and $Y$ are $\cC$-isomorphic, then $\mu(X)=\mu(Y)$ (a ``motivic'' property);  for $X,Y\subseteq M^m$ we can represent $X\cup Y$ as the disjoint union of $X\setminus Y, X\cap Y$, and $Y\setminus X$ to obtain
$$\mu(X\cup Y)+\mu(X\cap Y)\ =\ \mu(X)+\mu(Y).$$ 
If $f: X\to Y$ is a $\cC$-morphism, then $\mu(X)=\sum_a a\mu(Y_a)$, using  notation of (F2). 

\medskip\noindent
{\bf Example.} Let $\cM=\T$ as in the beginning of the introduction. Then the subsets of $\R^n$ definable in $\cM$ are  semialgebraic in the sense of $\R$, so the full subcategory $\cC$ of $\Def(\cM)$
whose objects are these semialgebraic sets satisfies conditions (a), (b), (c). The above Euler characteristic $Y\mapsto E(Y)$ is then a Fubini measure on $\cC$ with values in the ring $\Z$ of integers.  Also $Y\mapsto \dim(Y)$ is a Fubini measure on $\cC$, taking values in the tropical semiring $\N_{\trop}$, which has underlying set
$\N\cup\{-\infty\}$ and addition and multiplication given respectively by
$(p,q)\mapsto \max(p,q)$ and $(p,q)\mapsto p+q$; this semiring has $-\infty$ as its zero element for addition and $0$ as its multiplicative identity. 

\bigskip\noindent
To state the main result of this paper, we consider a $0$-definable set $C\subseteq M$ with more than one element, and let $\Def(C)$ be the full subcategory of $\Def(\cM)$ whose objects are the definable subsets of the powers $C^n$.  Below we define the notion, for definable $X\subseteq M^m$, of {\em $X$ being fiberable over $C$}. These $X$ are the objects of a full subcategory
$\Def(C)^{\f}\supseteq  \Def(C)$ of $\Def(\cM)$. The categories $\Def(C)$ and $\Def(C)^{\f}$ satisfy conditions (a), (b), (c)  imposed earlier on $\cC$. 

\begin{theorem}\label{thm} If $C$ is stably embedded in $\cM$ and $\mu$ an $A$-valued Fubini measure on $\Def(C)$,
then $\mu$ extends uniquely to an $A$-valued Fubini measure on $\Def(C)^{\f}$. 
\end{theorem} 

\noindent
Suppose $\cM$ is $\omega$-saturated, and let $X\subseteq M^m$ be definable. By recursion on $r\in \N$ we define ``$X$ is $r$-step fiberable over $C$":  this means the existence of a definable $f: X \to M^n$ such that 
if $r=0$, then $f$ is injective and $f(X)\subseteq C^n$, and if $r\ge 1$, then $f(X)$ and every fiber $f^{-1}(y)$, $y\in M^n$,
is $(r-1)$-step fiberable over $C$.

Now drop the assumption that $\cM$ is $\omega$-saturated, and let $X\subseteq M^m$ be definable.
Then the definition of ``$X$ is $r$-step fiberable
over $C$'' (with $\cM$ as ambient structure) is more complicated, see end of Section~\ref{proof}, but will be such that it is equivalent to ``$X^*$ is $r$-step fiberable  over $C^*$" with $\cM^*$ as ambient structure, where $\cM^*=(M^*,\dots)$ 
is any $\omega$-saturated elementary extension of $\cM$. Define $X$ to be {\em fiberable over $C$\/} if  it is $r$-step fiberable over $C$ for some $r$. 

Using this equivalence the proof of Theorem~\ref{thm} reduces to the case where $\cM$ is $\omega$-saturated, which is very convenient. The idea is to show by induction on $r$ that $\mu$ extends to an $A$-valued Fubini measure on the full subcategory of $\Def(\cM)$ whose objects are the definable $X\subseteq M^m$ that are $r$-step fiberable over $C$. In the step from $r$ to $r+1$ we consider an $(r+1)$-step fiberable $X\subseteq M^m$ over $C$. This is witnessed by a definable $f: X\to M^n$ such that $f(X)$ and all $f^{-1}(y)$ are $r$-step fiberable over $C$. A carefully chosen inductive assumption will tell us that $\mu\big(f(X)\big)$ is defined and that $\mu\big(f^{-1}(y)\big)$ is defined and takes only finitely many values for $y\in f(X)$. Then we define 
$\mu_f(X)$ as a suitably weighted sum of these values and show that $\mu_f(X)$ does not depend on $f$, but only on $X$, and that the resulting $\mu(X)=\mu_f(X)$ defined for such $X$ preserves the Fubini property. We also have to show that the inductive assumption is preserved in this step from $r$ to $r+1$. All this is carried out in detail in Section~\ref{proof}.  

Section~\ref{prelim} contains preliminary results. In section~\ref{co} we relate fiberability over $C$ to the more intrinsic notion of {\em co-analyzability relative to $C$\/}  from \cite{HHM}, and show their equivalence under various model theoretic conditions. These conditions are satisfied by $\T$ and by differentially closed fields, with $C$ their constant field. This is discussed in Section~\ref{special}, where it is used to answer questions left open in \cite{ADH}. In the short last section~\ref{ms} we state for the record a many-sorted version of Theorem~\ref{thm}. 

\medskip\noindent
The  measures from  \cite{MS, AMSW} also take values in semirings and satisfy {\em Fubini}, but seem to be of a rather different nature. Nevertheless, some variant of Theorem~\ref{thm} might hold in that setting. One could also take a look at motivic measures as in ~\cite{HK}, and consider their ``Fubini" nature, or lack of it. 

\medskip\noindent
I thank Ben Castle, Artem Chernikov, Rahim Moosa, Anand Pillay, and Charles Steinhorn for useful feedback after a talk I gave on an earlier version of this paper. 

\subsection*{Notations and Terminology} We let $d,e, k, l, m,n,r$, sometimes subscripted, range over $\N=\{0,1,2,\dots\}$.  For sets $P, Q$, and $\cal{S}\subseteq P\times Q$ and $p\in P$ we set $$\cal{S}(p)\ :=\ \{q\in Q:\ (p,q)\in \cal{S}\}. $$  (Think of $\cal{S}$ as describing the family $\big(\cal{S}(p)\big)_{p\in P}$ of subsets of $Q$.) Throughout, $\cM=(M;\dots)$ is a one-sorted $\cal{L}$-structure. When $C$ gets mentioned, we assume $C\subseteq M$ is $0$-definable and $|C|\ge 2$. Recall, for example from \cite{HHM}, that $C$ is said to be {\em stably embedded in $\cM$\/} if for every $0$-definable $\cal{S}\subseteq M^{m}\times M^n$ there is a $0$-definable $\cal{T}\subseteq M^{\mu}\times M^n$, $\mu\in \N$, such that for all $p\in M^m$ there exists $c\in C^{\mu}$ for which $\cal{S}(p)\cap C^n = \cal{T}(c)\cap C^n$. 
  For $\omega$-saturated $\cM$, this is equivalent to all definable $X\subseteq C^n$ being $C$-definable. 
 
Suppose $X\subseteq M^m$ and $Y\subseteq M^n$ are definable. Then a {\em definable family of maps from subsets of $X$ into $Y$\/} is a definable set
$\cal{F}\subseteq P\times X \times Y$, with definable $P\subseteq M^k$ such that for each ``parameter" $p\in P$ the set
$\cal{F}(p)\subseteq X\times Y$ is the graph of a map $f_p$ with domain a subset of $X$ and codomain $Y$.  We think of such $\cal{F}$ as the family $(f_p)_{p\in P}$. Abusing language we write $f\in \cal{F}$ to indicate that $f=f_p$ for some $p\in P$, and to state that the set of $f\in \cal{F}$ with a certain property is definable means that the set of $p\in P$ such that $f_p$ has that property is definable. Given $\cal{S}\subseteq P\times Q$, $p\in P$, and a map $f: \cal{S} \to Z$, we let $f(p,-)$ denote the map $q\mapsto f(p,q): \cal{S}(p)\to Z$. 

Our definition of ``$r$-step fiberable over $C$'' differs from that of ``fiberable by $C$ in $r$ steps'' 
in \cite{ADH} for a Liouville closed $H$-field $K$ and its constant field $C$. But, in that case,  ``fiberable over $C$'' as defined here is  equivalent to ``fiberable by $C$ in $r$ steps for some $r$" as defined in \cite{ADH}.  
%A map $f: X\to Y$ is said to be {\em finite-to-one} if $f^{-1}(y)$ is finite for all $y\in Y$; if $\cM$ is $\omega$-saturated and $X, Y, f$ are definable in $\cM$, this is equivalent to the existence of an $e$ such that $|f^{-1}(y)|\le e$ for all $y\in Y$. 

\section{Preliminaries}\label{prelim}

\noindent
It will be convenient to consider measures that do not necessarily satisfy (F3). We call  them {\em semi-Fubini measures\/} and they are defined in the first subsection below. In the second subsection we consider {\em Fubini maps}: morphisms that satisfy a version of (F3), but with respect to a semi-Fubini measure. The key fact about them is that a composition of Fubini maps is a Fubini map: Lemma~\ref{lf}.

\medskip\noindent
{\em In this section $\cC$ is a full subcategory of $\Def(\cM)$ satisfying {\rm (a), (b), (c)}, such that some $X\in \cC$ has more than one element}. Note that then for every $n$ some $X\in \cC$ has more than $n$ elements. Let  $\cC^{\iso}$ be the full subcategory of $\Def(\cM)$ whose objects are the definable sets $X\subseteq M^m$ that are isomorphic, in $\Def(\cM)$, to an object of $\cC$.  Then $\cC^{\iso}$ also satisfies (a), (b), (c): this is clear for (b) and (c), and to check (a), let $f: X'\to X$ and $g: Y'\to Y$ be isomorphisms in $\Def(\cM)$ with $X,Y\in  \cC$, and $X', Y'\subseteq M^m$; we need to show that then there is a definable bijection $X'\cup Y'\to Z$ with $Z\in \cC$. First arrange $X'$ and $Y'$ to be disjoint. The case that $X'$ or $Y'$ is empty is trivial, so assume $X'$ and $Y'$ are nonempty. Take $V\in \cC$ with distinct elements $v_1, v_2$, and pick $x_0\in X$ and $y_0\in Y$.
Then $ \{x_0\}\times Y\times \{v_1\}$ and $X\times \{y_0\}\times \{v_2\}$ are disjoint subsets of $X\times Y\times V\in \cC$, their
union $Z$ is a definable subset of $X\times Y\times V$, and thus $Z\in  \cC$. Then $h: X'\cup Y' \to Z$ with $h(x)=\big(f(x),y_0, v_2\big)$ for $x\in X'$ and $h(y)=\big(x_0,g(y), v_1\big)$ for $y\in Y'$ is a definable bijection as required. Note that any finite subset of any $M^n$ is an object of $\cC^{\iso}$, and that $\big(\cC^{\iso})^{\iso}=\cC^{\iso}$.

{\em Throughout, $A$ is a commutative semiring, and we let $a,b,c$ range over $A$}. We consider $\N$ as a semiring in the obvious way, and observe that there is a unique semiring morphism $\N\to A$; we set $da:=d_Aa$ where $d_A$ is the image of $d$ in $A$ under this semiring morphism.

\subsection*{Semi-Fubini measures}  Let $\mu$ be an {\em $A$-valued semi-Fubini measure on $\cC$}, that is, $\mu:  \cC\to A$ satisfies conditions (d), (e) in the definition of ``Fubini measure'', every $\cC$-morphism $f: X\to Y$ satisfies (F1) and (F2), and 
$\mu(X)=\mu(Y)$ whenever $X$ and $Y$ are $\cC$-isomorphic. Consequence: for $X,Y\in \cC$ and definable $\cal{S}\subseteq X\times Y$ the set $\{\mu\big(\cal{S}(x)\big):\ x\in X\}$ is finite and for all $a$ the set $\{x\in X:\ \mu\big(\cal{S}(x)\big)=a\}$ is definable; this follows by applying (F1) and (F2) to the projection map $\cal{S} \to X$.

We first  extend $\mu$ to $ \cC^{\iso}$ by $\mu(X'):=\mu(X)$ whenever $X'\in \Def(\cM)$ is isomorphic to $X\in \cC$. It is easy to check that this extended $\mu$ is then an $A$-valued semi-Fubini measure on $\cC^{\iso}$: for clause (e) we deal with $\mu(X'\cup Y')$ as in the earlier $X'\cup Y'$ argument for $\cC^{\iso}$. 
If $\mu:\cC\to A$ is a Fubini measure, then so is its extension $\mu:  \cC^{\iso}\to A$.

\medskip\noindent
Let $f: X\to Y$ be a $\Def(\cM)$-morphism. Call $f$ {\em uniform over $(\cC,\mu)$\/} if \begin{enumerate}
\item $Y\in  \cC$ and $f^{-1}(y)\in \cC^{\iso}$ for all $y\in Y$;
\item for all definable $X'\subseteq X$  the subset $\{\mu\big(f^{-1}(y)\cap X'\big):\ y\in Y\}$ of $A$ is finite;
\item for all definable $X'\subseteq X$ and all $a$ the subset
$$Y_{f, X',a}\  :=\  \{y\in Y:\ \mu\big(f^{-1}(y)\cap X'\big)=a\}$$
 of $Y$  is definable.
\end{enumerate} 
In connection with (3) we write $Y_{f,a}$ instead of $Y_{f,X,a}$, and even $Y_a$ instead of $Y_{f,a}$  when $f$ is clear from the context. Note that if $f$ is uniform over $(\cC,\mu)$, then so is any restriction
$f|_{X'}: X'\to Y'$ with definable $X'\subseteq X$ and $Y'\subseteq Y$ such that $f(X')\subseteq Y'$, and that $Y_a$ is empty for all but finitely many $a$. 

{\em Examples}.  Any $\cC$-morphism is uniform over $(\cC,\mu)$.  Suppose the $\Def(\cM)$-morphism $f: X\to Y$ is {\em uniformly finite over $\cC$\/}, that is,  $Y\in \cC$ and for some $e$ we have $|f^{-1}(y)|\le e$ for all $y\in Y$. Then $f$ is uniform over $(\cC,\mu)$. 

\medskip\noindent
{\em In the rest of this subsection we assume the $\Def(\cM)$-morphism $f: X\to Y$ is uniform over $(\cC,\mu)$}. Then $Y_{f,a}$ is empty for all but finitely many $a$, and we set
$$\mu_f(X)\ :=\ \sum_a a\mu(Y_{f,a})\in A.$$
{\em Example\/}:  for any $Y\in \cC$ the identity map $\operatorname{id}: Y\to Y$ is uniform over $(\cC,\mu)$ with 
$\mu_{\operatorname{id}}(Y)=\mu(Y)$. 

\medskip\noindent
Once checks easily that if $Y$ is the disjoint union of definable subsets $Y_1, Y_2$, then 
 $\mu_f(X)=\mu_{f_1}(X_1)+\mu_{f_2}(X_2)$, where
$$X_i\ :=\ f^{-1}(Y_i),\qquad  f_i\ :=\  f|_{X_i}: X_i\to Y_i\  \text{ for }i=1,2.$$
The next observation takes a bit more work:

\begin{lemma}\label{1} Let  $X$ be the disjoint union of definable subsets $X_1, X_2$, and set $f_i:= f|_{X_i}: X_i\to Y$ for $i=1,2$. Then 
$\mu_f(X)=\mu_{f_1}(X_1)+\mu_{f_2}(X_2)$. 
\end{lemma} 
\begin{proof} We have $\mu_{f_1}(X_1)=\sum_b b\mu(Y_{1,b})$ and $\mu_{f_2}(X_2)=\sum_c c\mu(Y_{2,c})$ where
$$Y_{1,b} := \{y\in Y:\ \mu\big(f^{-1}(y)\cap X_1\big)=b\},\qquad Y_{2,c} := \{y\in Y:\ \mu\big(f^{-1}(y)\cap X_2\big)=c\}.$$
 Set $Y_{b,c} :=  Y_{1,b}\cap Y_{2,c}$, so $Y$ is the disjoint union of the $Y_{b,c}$, and $Y_{b,c}$ is empty for all but finitely many $(b,c)$. Moreover,
$Y_{1,b}=\bigcup_c Y_{b,c}$, $Y_{2,c}=\bigcup_b Y_{b,c}$, and $Y_a=\bigcup_{b+c=a} Y_{b,c}$ (disjoint unions). Hence 
$$\mu_{f_1}(X_1)\ =\ \sum_{b}b\mu(Y_{1,b})\ =\ \sum_b b\sum_{c}\mu(Y_{b,c})\ =\ \sum_{b,c}b\mu(Y_{b,c}),$$
and likewise $\mu_{f_2}(X_2)=\sum_{b,c}c\mu(Y_{b,c})$. Hence
\begin{align*} \mu_f(X)\ &=\ \sum_a a\mu(Y_a)\ =\ \sum_a a\sum_{b+c=a}\mu(Y_{b,c})\ =\  \sum_{b,c}(b+c)\mu(Y_{b,c})\\
\ &=\ \sum_{b,c}b\mu(Y_{b,c})+\sum_{b,c}c\mu(Y_{b,c})\ =\ \mu_{f_1}(X_1)+\mu_{f_2} (X_2).\quad \qedhere
\end{align*} 
\end{proof} 

\subsection*{Fubini maps} As in the previous subsection, {\em $\mu:  \cC\to A$ is a semi-Fubini measure on $\cC$}. Here is a consequence of Lemma~\ref{1}:

\begin{lemma}\label{f1} Let $f: X\to Y$ be a $\cC$-morphism, with $X$ the disjoint union of definable $X_1,X_2\subseteq X$.
Suppose for $i=1,2$ that $\mu(X_i)=\sum_a a\mu(Y_{i,a})$ where
$Y_{i,a}=\{y\in Y:\ \mu\big(f^{-1}(y)\cap X_i\big)=a\}$. Then
$\mu(X)= \sum_a a\mu(Y_a)$.
\end{lemma} 
\begin{proof} The hypotheses of the lemma amount to  $\mu(X_i)=\mu_{f_i}(X_i)$ for $i=1,2$ where $f_i:=f|_{X_i}: X_i\to Y$. 
Hence $$\mu(X)\ =\ \mu(X_1)+\mu(X_2)\ =\ \mu_{f_1}(X_1)+ \mu_{f_2}(X_2)\ =\ \mu_f(X)\ =\ \sum_a a\mu(Y_a), $$
where we use Lemma~\ref{1} for the third equality. 
\end{proof} 

\noindent
A $\cC$-morphism
$f: X\to Y$ is said to be a {\em Fubini map} (with respect to $(\cC,\mu)$) if for all definable $X'\subseteq X$ and $Y'\subseteq Y$ with $f(X')\subseteq Y'$ and
such that $\mu\big(f^{-1}(y)\cap X'\big)$ takes the constant value $a$ on $Y'$ we have $\mu(X')=a\mu(Y')$. Note that if the $\cC$-morphism $f: X\to Y$ is a Fubini map, then $\mu(X)=\sum_a a\mu(Y_a)$, and for all definable $X'\subseteq X$ and $Y'\subseteq Y$ with $f(X')\subseteq Y'$ the restriction $f|_{X'}: X'\to Y'$ is also a Fubini map.

\begin{lemma} \label{lf} Suppose the $\cC$-morphisms $f: X\to Y$ and $g: Y\to Z$ are  Fubini maps. Then $g\circ f: X\to Z$ is a Fubini map. 
\end{lemma} 
\begin{proof}  For any definable $X'\subseteq X$ and $Z'\subseteq Z$ with $(g\circ f)(X')\subseteq Z'$, the restriction
of $g\circ f$ to a map $X'\to Z'$ is a composition $g'\circ f'$ where $f'$ is the restriction of $f$ to a map $X'\to f(X')$ and
$g'$ is the restriction of $g$ to a map $f(X')\to Z'$. Since these restrictions of $f$ and $g$ are also Fubini maps, it suffices to derive from the hypothesis of the lemma that $\mu(X)=\sum_c c\mu(Z_c)$, $Z_c:=\{z\in Z:\  \mu\big((g\circ f)^{-1}(z)\big)=c\}$. 
Take distinct $a_1,\dots, a_n\in A$ such that $Y=\bigcup_{i=1}^n Y_{a_i}$ and set $X_i:= f^{-1}(Y_{a_i})$, so $X$ is the disjoint union of $X_1,\dots, X_n$. By Lemma~\ref{f1}, with $n$ parts of $X$ instead of two, it suffices to show that for $i=1,\dots,n$ we have $\mu(X_i)=\sum_c c \mu(Z_{i,c})$ where
$Z_{i,c}=\{z\in Z:\ \mu\big((g\circ f)^{-1}(z)\cap X_i\big)=c\}$. In other words, we can reduce to considering instead of
$g\circ f$ its restrictions $g_i\circ f_i: X_i\to Z$ where $f_i=f|_{X_i}: X_i\to Y_{a_i}$ and $g_i=g|_{Y_{a_i}}: Y_{a_i}\to Z$.
Fixing one $i\in \{1,\dots,n\}$ and renaming $X_i, Y_{a_i}, f_i, g_i, a_i$ as $X, Y, f, g,a$ we have reduced to the case that $\mu\big(f^{-1}(y)\big)$ takes the constant value $a$ on $Y$, so $\mu(X)=a\mu(Y)$ (since $f$ is still a Fubini map). 

  Set $Z_{g,b}\ :=\ \{z\in Z:\ \mu\big(g^{-1}(z)\big)=b\}$. 
Then $Z_{g,b}$ is empty for all but finitely many $b$ and $\mu(Y)=\sum_b b\mu(Z_{g,b})$, since $g$ is a Fubini map, so
$$\mu(X)\ =\ a\mu(Y)\ =\ a\sum_b b\mu(Z_{g,b})\big)\ =\ \sum_{b}ab\mu(Z_{g,b})\ \ =\ \sum_c c\big(\sum_{ab=c}\mu(Z_{g,b})\big).$$
Thus to obtain $\mu(X)=\sum_c c\mu(Z_c)$ it is enough to show that $Z_c=\bigcup_{ab=c} Z_{g,b}$.

\medskip\noindent
Put $Y_{b}:=g^{-1}(Z_{g,b})$ and $X_{b}:=f^{-1}(Y_{b})$,
and set 
$$f_{b}\ :=\ f|_{X_{b}}\ :\  X_{b}\to Y_{b},\qquad  g_{b}\ :=\ g|_{Y_{b}}: Y_{b}\to Z_{g,b}.$$Now
$\mu\big(g_{b}^{-1}(z)\big)=b$ for $z\in Z_{g,b}$, and $\mu\big(f_{b}^{-1}(y)\big)=a$ for $y\in Y_{b}$. Thus given $z\in Z_{g,b}$,  $\mu\big(f_{b}^{-1}(y)\big)$ takes the constant value $a$ on $g_{b}^{-1}(z)$, so $f$ being a Fubini map we obtain for $z\in Z_{g,b}$:
$$\mu\left(f_{b}^{-1}\big(g_{b}^{-1}(z)\big)\right)\ =\ a\mu(g_{b}^{-1}(z))\ =\ ab.$$
Now $(g\circ f)^{-1}(z)=f_{b}^{-1}\big(g_{b}^{-1}(z)\big)$ for $z\in Z_{g,b}$. Hence, if $ab=c$, then $Z_{g,b}\subseteq Z_c$, and so if $ab\ne c$, then $Z_{g,b}$ is disjoint from $Z_c$. Thus indeed $Z_c=\bigcup_{ab=c} Z_{g,b}$.
\end{proof}

\subsection*{Three observations} The first two lemmas below will be used in Section~\ref{special}. 

\begin{lemma}\label{smfub} If $\cM$ is strongly minimal $($which includes $M$ being infinite$)$, then $X\mapsto \operatorname{MR}(X)$ is an $\N_{\trop}$-valued Fubini measure on $\Def(\cM)$. 
\end{lemma} 

\noindent
This is well-known and just stated for the record. 

\begin{lemma}\label{munu} Suppose $\mu,\nu$ are $A$-valued Fubini measures on $\cC:=\Def(C)$, and $\mu(Y)=\nu(Y)$ for all definable $Y\subseteq C$. Then $\mu=\nu$.
\end{lemma} 
\begin{proof} By induction on $n$, show that $\mu(Y)=\nu(Y)$ for definable $Y\subseteq C^n$. For the inductive step and definable $Y\subseteq C^{n+1}$, use ``Fubini'' and the projection map $$\quad\qquad\qquad \qquad\qquad (y_1,\dots, y_{n+1})\mapsto (y_1,\dots, y_n): Y\to C^n.\qquad\qquad\qquad\qquad \qedhere $$
\end{proof}

\begin{lemma}  If $A, B$ are commutative semirings and $\mu: \cC\to A$ and $\nu: \cC\to B$ are Fubini measures, then $(\mu,\nu): \cC\to A\times B$ is a Fubini measure. 
\end{lemma}

\section{Proof of the Theorem}\label{proof}

\noindent
{\it In this section $\mu$ is an $A$-valued Fubini measure on $\Def(C)$}. Till further notice we also assume $\cM$ is $\omega$-saturated, indicating at the end of the section how to reduce to that case. 

\medskip\noindent
By recursion on $r$ we now introduce a sequence $(\cC_r)$ of full subcategories of $\Def(\cM)$, such that $ \cC_r\subseteq \cC_{r+1}$ for all $r$ and $\bigcup_r \cC_r=\Def(C)^{\f}$:

\medskip\noindent
 $\cC_0:= \Def(C)^{\iso}$, and the objects of $\cC_{r+1}$ are the definable $X\subseteq M^m$ for which there exists a definable
$f: X\to Y$ with $Y\in \cC_r$, such that $f^{-1}(y)\in  \cC_r$ for all $y\in Y$; such $f$ is called a $\cC_{r+1}$-witness. Thus the objects of $\cC_r$ are the definable $X\subseteq M^m$ that are $r$-step fiberable over $C$ as defined in the Introduction.
In particular, the objects of $\cC_0$ are the definable $X\subseteq M^m$ that admit a definable injective map $X\to C^n$. Also, if the definable set $X\subseteq M^m$ admits a definable map $X\to C^n$ all whose fibers are finite, then $X\in \cC_1$.  

\medskip\noindent
 Induction on $r$ shows that $\cC_r$ satisfies conditions (a), (b), (c) for all $r$, and that $\cC_r^{\iso}=\cC_r$ for all $r$.  In the inductive step we note that if $f: X\to Y$ is a $\cC_{r+1}$-witness and $X'\subseteq X$ and $Y'\subseteq Y$ are definable with
$f(X')\subseteq Y'$, then $f|_{X'}: X'\to Y'$ is also a $\cC_{r+1}$-witness. Induction on $r$ also shows that there is at most one
$A$-valued Fubini measure on $\cC_r$ extending the given $A$-valued Fubini measure $\mu$ on $\Def(C)$. In the beginning of the previous section we have extended $\mu$ to an $A$-valued Fubini measure on $\cC_0$, also to be denoted by $\mu$.

\subsection*{More about $\cC_0$} {\it In this subsection we assume that $C$ is stably embedded in $\cM$}.
Let  $\cal{S}\subseteq M^k\times C^n$ be definable. Think of $\cal{S}$ as the definable family of sets
$\cal{S}(p)\subseteq C^n$ parametrized by the points $p\in M^k$. 

\begin{lemma}\label{l1}  The set $\{\mu\big(\cal{S}(p)\big):\ p\in M^k\}$ is finite, and for every $a$ the set $\{p\in M^k:\ \mu\big(\cal{S}(p)\big)=a\}$ is definable.
\end{lemma} 
\begin{proof} Since $\cM$ is $\omega$-saturated and $C$ is stably embedded in $\cM$, we have definable
$\cal{T}\subseteq C^m\times C^n$ such that for every $p\in M^k$ there is  $x\in C^m$ with  $\cal{S}(p)=\cal{T}(x)$. 
Now apply (F1) and (F2) to the projection map $\cal{T}\to C^m$ and use that its fibers are $\cC_0$-isomorphic to the sections 
$\cal{T}(x)$. 
\end{proof}

\noindent
Let $\cal F$ be a definable family of {\em injective\/} maps from subsets of $M^m$ to $C^n$.  Then $\mu\big(f^{-1}(y)\big)\in \{0,1\}\subseteq A$ for $(f,y)\in \cal F \times C^n$.
For $f\in \cal F$, set
$$D(f)\ :=\ \text{domain}(f)\subseteq M^m, \qquad R(f):= f\big(D(f)\big)\subseteq C^n,$$
so $D(f)\in \cC_0$ and $\mu\big(D(f)\big)=\mu\big(R(f)\big)$.   The $R(f)$ with $f\in \cal F$ are the members of a definable family of subsets of $C^n$, so by Lemma~\ref{l1}:

\begin{lemma} \label{l2} The set $\{\mu\big(D(f)\big):\ f\in \cal F\}$ is finite
and for every $a\in A$ the set $\{f\in \cal F:\ \mu\big(D(f)\big)=a\}$
is definable. 
\end{lemma}

\noindent
To start our inductive proof of Theorem~\ref{thm} we shall need:

\begin{lemma}\label{lind} Let $\cal{S}\subseteq M^k\times M^m$ be definable such that $\cal{S}(p)\in  \cC_0$ for all $p\in M^k$. Then the set 
$\{\mu\big(\cal{S}(p)\big):\ p\in M^k\}$ is finite, and the set $\{p\in M^k:\ \mu\big(\cal{S}(p)\big)=a\}$ is definable for every $a$.
\end{lemma}
\begin{proof}  Augmenting the language $\cal{L}$ of $\cM$ by names for finitely many elements of $M$ and expanding $\cM$ accordingly we arrange that $\cal{S}$ is $0$-definable.
For each $p\in M^k$ we have a definable injective map $f_p: \cal{S}(p)\to C^{n(p)}$ with $n(p)\in \N$.  So for each $p\in M^k$ there are $d,n$ and a $0$-definable $\cal{T}\subseteq M^{d+k+m+n}$ such that 

\medskip
(*)  for some $x\in M^d$ the set $\cal{T}(x,p)\subseteq M^{m+n}$ is the graph of an injective map $$f_p\ :\  \cal{S}(p)\to C^n.$$ 
By $\omega$-saturation this gives a finite set $E$ of tuples $(d,n,\cal{T})$ as above such that for each $p\in M^k$ some
tuple $\tau=(d,n,\cal{T})\in E$ satisfies (*).
Each $\tau=(d,n,\cal{T})\in E$ yields a definable family $\cal{F}_\tau$ of injective maps from subsets of $M^m$ into $C^n$ such that  for each
$f\in \cal{F}_\tau$ the domain of $f$ is of the form $\cal{S}(p)$ with $p\in M^k$, and for each $p\in M^k$ there is a $\tau\in E$ and an $f\in \cal{F}_\tau$ with domain $\cal{S}(p)$. It remains to appeal to Lemma~\ref{l2}. 
\end{proof}

\subsection*{Going beyond $\cC_0$} {\it In this subsection $C$ is stably embedded in $\cM$.} 
We now make the following  inductive assumption (which is satisfied for $r=0$ by Lemma~\ref{lind}): \begin{enumerate}
 \item[(IA1)] the $A$-valued Fubini measure $\mu$ on $\cC_0$ extends to an $A$-valued Fubini measure on $\cC_r$; we denote such a (necessarily unique) extension also by $\mu$;
 \item[(IA2)] for all definable $\cal{S}\subseteq M^k\times M^m$ with $\cal{S}(p)\in  \cC_r$ for all $p\in M^k$, the set $\{\mu\big(\cal{S}(p)\big):\ p\in M^k\}$ is finite, and the set $\{p\in M^k:\ \mu\big(\cal{S}(p)\big)=a\}$ is definable for every $a$.  
\end{enumerate} 
In particular, for every definable $X\subseteq M^m$ and $\cC_{r+1}$-witness $f: X\to Y$ the set 
 $\{\mu\big(f^{-1}(y)\big):\ y\in Y\}$ is finite, and $Y_{f,a}:=\{y\in Y:\ \mu\big(f^{-1}(y)\big)=a\}$ is definable for every $a$. It follows that such $f$ is uniform over $(\cC_r,\mu)$, and so $\mu_f(X)$ is defined, and Lemma~\ref{1} and the remark preceding it are applicable. 
 
 In the next lemma and corollary, $X\subseteq M^m$ is definable and $f: X\to Y$ is a $\cC_{r+1}$-witness. The goal here is to show that $\mu_f(X)$ does not depend on $f$.

 \begin{lemma}\label{l0++} Let $\phi: Y\to W$ be a $\cC_r$-morphism such that $h:=\phi\circ f: X\to W$ is also a $\cC_{r+1}$-witness. Then $\mu_h(X)=\mu_f(X)$.
 \end{lemma}
 \begin{proof} The set $Y$ is the disjoint union of its definable subsets $Y_{f,a}$, and $Y_{f,a}$ is empty for all but finitely many $a$.  By Lemma~\ref{1} and the remark preceding it we can focus on one such $Y_{f,a}$, replace $f$ by its restriction to $f^{-1}(Y_{f,a})$ with $Y_{f,a}$ as codomain, and $\phi$ by its restriction to $Y_{f,a}$. Renaming then $f^{-1}(Y_{f,a})$ and $Y_{f,a}$ as $X$ and $Y$ we arrange $\mu\big(f^{-1}(y)\big)=a$ for all $y\in Y$. Hence $\mu_f(X)=a\mu(Y)$. 

If $w\in W_{\phi,b}$, then $h^{-1}(w)\in  \cC_r$, $f|_{h^{-1}(w)}:h^{-1}(w)\to Y$ is a $\cC_r$-morphism, and
 $h^{-1}(w)=f^{-1}\big(\phi^{-1}(w)\big)$, so $\mu\big(h^{-1}(w)\big)=a\mu\big(\phi^{-1}(w)\big)=ab$. Hence 
$$\mu_h(X)\ =\ \sum_b ab\mu(W_{\phi,b})\ =\ a\sum_b b\mu(W_{\phi,b})\ =\ a\mu(Y),$$
using ``Fubini'' for $\phi$ in the last step.  Thus $\mu_h(X)=\mu_f(X)$.
 \end{proof} 
 
\begin{cor}\label{c0++} Suppose the definable map $g: X \to Z$ is also a $\cC_{r+1}$-witness. Then 
$\mu_f(X) = \mu_g(X)$.
\end{cor} 
\begin{proof}  Set $D:=\{\big(f(x), g(x)\big):\  x\in X\}$, a definable subset of $Y\times Z$. The definable map
$(f,g): X \to D$ is a $\cC_{r+1}$-witness.  For the  projection maps $\pi_1: D \to Y$ and
$\pi_2: D \to Z$ we have $f=\pi_1\circ (f,g)$ and $g=\pi_2\circ (f,g)$. By Lemma~\ref{l0++}  this yields $\mu_f(X)=\mu_{(f,g)}(X)=\mu_g(X)$. \end{proof}  

\noindent
This allows us to define $\mu:  \cC_{r+1}\to A$ by $\mu(X):=\mu_f(X)$, where $f: X \to Y$ is any $\cC_{r+1}$-witness. This function $\mu:  \cC_{r+1}\to A$ extends the Fubini measure $\mu$ on $\cC_r$  and satisfies conditions {\rm (d), (e)}. 
Towards proving that $\mu:  \cC_{r+1}\to A$ is a Fubini measure on $\cC_{r+1}$ we first show that it is a semi-Fubini measure.

Let $Y\in  \cC_r$, and
let $\cal F$ be a definable family of maps from subsets of $M^m$ to $Y$ such that $f^{-1}(y)\in \cC_r$ for all $(f,y)\in \cal F \times Y$; so each $f\in \cal F$ is a $\cC_{r+1}$-witness.
By (IA2), $\mu\big(f^{-1}(y)\big)$ takes only finitely many values for $(f,y)\in \cal F \times Y$, and
 $\{(f,y):\in \cal F\times Y:\  \mu\big(f^{-1}(y)\big)=a\}$ is definable, for every $a$. 
For $f\in \cal F$, set
$$D(f)\ :=\ \text{domain}(f)\subseteq M^m, \qquad R(f)_a\ :=\ \{y\in Y:\ \mu\big(f^{-1}(y)\big)=a\},$$
so $D(f)\in \cC_{r+1}$ with $\mu\big(D(f)\big)=\mu_f\big(D(f)\big)=\sum_a a\mu\big(R(f)_a\big)$.  Also the set
$$\{a:\ R(f)_a\ne \emptyset \text{ for some }f\in \cal F\}$$
is finite, and so is $\{\mu\big(R(f)_a\big):\ (f,a)\in \cal F\times A\}$ by (IA2),
since the $R(f)_a$ with $(f,a)\in \cal F\times A$ are the members of a definable family of subsets of $Y$. By this uniform definability of the $R(f)_a$ with only finitely many relevant $a$, we obtain: 

\begin{lemma} \label{l2+} The set $\{\mu\big(D(f)\big):\ f\in \cal F\}$ is finite
and for every $a\in A$ the set $\{f\in \cal F:\ \mu\big(D(f)\big)=a\}$
is definable. 
\end{lemma}

\begin{lemma}\label{l3+}  Let $\cal{S}\subseteq M^k\times X$ be definable with $X\in  \cC_{r+1}$. Then the set $$\{\mu\big(\cal{S}(p)\big):\ p\in M^k\}$$ is finite, and for every $a$ the set $\{p\in M^k:\ \mu\big(\cal{S}(p)\big)=a\}$ is definable.
\end{lemma} 
\begin{proof}  Take a $\cC_{r+1}$-witness $g: X\to Y$ and apply Lemma~\ref{l2+} to the family of maps
$g|_{\cal{S}(p)}: \cal{S}(p)\to Y$, definably parametrized by the points $p\in M^k$. 
\end{proof} 
 
\noindent
Thus by Lemma~\ref{l3+}:

\begin{cor}\label{cor++} The function $\mu: \cC_{r+1}\to A$ is a semi-Fubini measure on $\cC_{r+1}$.
\end{cor} 

\noindent
Proving that $\mu: \cC_{r+1}\to A$ is a Fubini measure on $\cC_{r+1}$ now reduces to showing that every $\cC_{r+1}$-morphism
$f: X\to Y$ is a Fubini map with respect to $(\cC_{r+1},\mu)$. First we note that any $\cC_{r+1}$-witness
$f: X\to Y$ is a Fubini map with respect to $(\cC_{r+1},\mu)$ by the very definition of $\mu$ on $\cC_{r+1}$. More generally:

\begin{lemma}\label{ll0} Let $Y\in \cC_r$ and let $f: X\to Y$ be a $\cC_{r+1}$-morphism. Then $f$ is a Fubini map.
\end{lemma} 
\begin{proof} Take a $\cC_{r+1}$-witness $g: X \to Z$. Then $(f,g): X\to Y\times Z$ is a $\cC_{r+1}$-witness, hence a Fubini map, and so $f=\pi\circ (f,g)$ is a Fubini map by Lemma~\ref{lf}, since the natural projection map $\pi: Y\times Z\to Y$ is a $\cC_r$-morphism, hence a Fubini  map. 
\end{proof}

\begin{prop}\label{prop+} The function $\mu : \cC_{r+1}\to A$ is a Fubini measure on $\cC_{r+1}$.
\end{prop}
\begin{proof} 
Let $f: X\to Y$ be a $\cC_{r+1}$-morphism; it is enough to show that then $f$ is a Fubini map.
We reduce to the case that $\mu\big(f^{-1}(y)\big)$ takes the constant value $a$ on $Y$. We claim that then $\mu(X)=a\mu(Y)$.
Take a $\cC_{r+1}$-witness $\phi: Y\to Z$, set $h:=\phi\circ f$, and reduce further to the case that $\mu\big(\phi^{-1}(z)\big)$ takes the constant value $b$ on $Z$; then $\mu(Y)=b\mu(Z)$.
Let $z\in Z$. Then $\phi^{-1}(z)\in \cC_r$, so $f|_{h^{-1}(z)}: h^{-1}(z)\to \phi^{-1}(z)$ is a Fubini map by Lemma~\ref{ll0}. Hence 
$\mu\big(h^{-1}(z)\big)=a\mu\big(\phi^{-1}(z)\big)=ab$.  But $h: X\to Z$ is also a Fubini map by Lemma~\ref{ll0}, so $\mu(X)=ab\mu(Z)=a\mu(Y)$, as promised.
\end{proof}

\noindent
We have now shown that (IA1) holds with $r+1$ instead of $r$. It remains to show that (IA2) is inherited in going from $r$ to $r+1$.

\subsection*{Completing the inductive step}  We need a more explicit witnessing scheme for being in $\cC_r$ in order to complete the inductive step, and also for the later reduction to $\omega$-saturated $\cM$. (For now we do not impose $\omega$-saturation, stably embeddedness, or inductive hypotheses (IA1), (IA2).) 
 
 \medskip\noindent
 Let $X\subseteq M^m$ be definable. We need to extend the notion of $X$ being $r$-step fiberable over $C$
  in such a way that for $\omega$-saturated $\cM$ it is equivalent to the already given notion. The idea is to require for $r\ge 1$ a $\cC_r$-witness $f: X \to  f(X)\subseteq M^n$, and 
 $\cC_{r-1}$-witnesses for the fibers $f^{-1}(y)$ with $y\in f(X)$ to come from a definable family of maps, and a $\cC_{r-1}$-witness for $f(X)$.
 It is convenient to combine the latter into $\cC_{r-1}$-witnesses for the cartesian products $f^{-1}(y)\times f(X)$. This leads to a rather tricky recursive definition for a tuple of maps and natural numbers to be an {\em $r$-step fibering of $X$ over $C$}. To make this perhaps more palatable to the reader we first give the definitions for $r=0, 1, 2$, and then for general $r\ge 1$. 

\medskip\noindent
A {\em $0$-step fibering of $X$ over $C$\/} is a pair $(f_1; n_1)$ where $f_1: X_1\to M^{n_1}$ with $X_1=X$ is an injective definable map such that  $f_1(X_1)\subseteq C^{n_1}$. 

\medskip\noindent
A {\em $1$-step fibering of $X$ over $C$\/} is a tuple $\big(f_1, f_2; m_1,n_1, n_2\big)$
where $$f_1\ :\  X_1 \to  M^{n_1}, \qquad f_2\ :\  X_2\to M^{n_2}$$ 
with $X_1=X\subseteq M^m$ and $X_2\subseteq M^{m_1}\times M^{m+n_1}$ are definable maps such that for all $y\in f_1(X_1)$, some $x\in M^{m_1}$ yields a $0$-step fibering $\big(f_2(x,-); n_2\big)$
of the subset $f_1^{-1}(y)\times f_1(X_1)$ of $M^{m+n_1}$ over $C$. (Note: the $x\in M^{m_1}$ serve here as parameters for the definable family of witnesses for the sets $f_1^{-1}(y)\times f_1(X_1)$, and instead of the ambient space $M^m$ for $X$ we have the ambient space $M^{m+n_1}$ for these sets.) 

\medskip\noindent
A {\em $2$-step fibering of $X$ over $C$\/} is a tuple $\big(f_1, f_2, f_3; m_1, m_2, n_1, n_2, n_3\big)$ where
$$f_1: X_1\to  M^{n_1},\quad f_2: X_2\to M^{n_2},\quad 
f_3:  X_3 \to M^{n_3} $$
with $X_1=X\subseteq M^m$,  $X_2\subseteq M^{m_1}\times M^{m+n_1}$, $X_3\subseteq M^{m_1}\times M^{m_2}\times M^{m+n_1+n_2}$ 
are definable maps such that for all $y\in f_1(X_1)$, some $x\in M^{m_1}$ yields a $1$-step fibering
$$\big(f_2(x,-), f_3(x,-); m_2, n_2, n_3\big)$$ of the set
$f_1^{-1}(y)\times f_1(X_1)\subseteq M^{m+n_1}$ over $C$.

\medskip\noindent
For any $r\ge 1$, an
{\em $r$-step fibering of $X$ over $C$\/} is a tuple $$\big(f_1,\dots, f_r, f_{r+1};  m_1,\dots, m_r, n_1,\dots, n_{r+1}\big)$$  where for $j=1,\dots,r+1$:
$f_j: X_j\to M^{n_j}$ is a definable map with definable domain $X_j\subseteq  M^{m_1}\times \cdots \times M^{m_{j-1}}\times 
M^{m+n_1+\cdots + n_{j-1}}$
 such that $X_{1}=X$ (so $f_1: X\to M^{n_1}$), and for all $y\in f_1(X_1)$, some $x\in M^{m_1}$ yields a tuple
 $$\big(f_2(x,-),\dots, f_{r+1}(x,-); m_2,\dots, m_r, n_2,\dots, n_{r+1} \big)$$
 that is an $(r-1)$-step fibering of $f_{1}^{-1}(y)\times f_1(X_1)\subseteq M^{m+n_1}$ over $C$. Note that for $x\in M^{m_1}$
 and $j=2,\dots, r+1$ the domain of $f_j(x,-)$ is the set 
 $$X_j(x)\ \subseteq\ M^{m_2}\times \cdots\times M^{m_{j-1}}\times M^{(m +n_1)+(n_2+\cdots + n_{j-1})}.$$
%and $m+[j,n]=m_{+}+[j-1,n_{-}]$ where $m_{+}:= m+n_1$ and  $[j-1,n_{-}]:=\sum_{i=2}^{j-1}  n_i$.)

 \medskip\noindent
{\em Uniform definability}.  
Fix $k, m, m_1,\dots, m_r, n_1,\dots, n_{r+1}$. For $j=1, \dots, r+1$, let a $0$-definable set
$\Sigma_j\subseteq M^k\times M^{m_1}\times\cdots\times M^{m_{j-1}}\times M^{m+n_1+\cdots+ n_{j-1}}$ and a $0$-definable map
$f_j: \Sigma_j\to M^{n_j}$ be given. Then the set $P$ of $p\in M^k$ such that
$$\big(f_1(p,-),\dots, f_{r+1}(p,-); m_1,\dots, m_r, n_1,\dots, n_{r+1}\big)$$
is an $r$-step fibering of the domain $\Sigma_1(p)\subseteq M^m$ of $f_1(p,-)$ over $C$ is $0$-definable, and $P$ has
 a defining $\cal{L}$-formula that depends only on 
 $(k, m,m_1,\dots, m_r, n_1,\dots, n_{r+1})$ and given defining $\cal{L}$-formulas for $C$ and $f_1,\dots, f_{r+1}$, not on $\cM$. 
 
 \medskip\noindent
 {\em Restricting a fibering}. Let $\big(\vec{f}; \vec{m}, \vec{n}\big)$ with
$$\vec{f}=(f_1,\dots, f_{r+1}),\quad \vec{m}=(m_1,\dots, m_r),\  \vec{n}=(n_1,\dots, n_{r+1})$$
 be an $r$-step fibering of $X\subseteq M^m$ over $C$. Let $X'\subseteq X$ be definable.  We define by recursion on $r$  {\em the restriction of
$\big(\vec{f}; \vec{m}, \vec{n}\big)$ to $X'$}. It will be an $r$-step fibering 
$$\big(\vec{f}'; \vec{m}, \vec{n}\big),\qquad \vec{f}'=(f'_1,\dots, f_{r+1}')$$ of $X'$ over $C$ such that
for $j=1,\dots, r+1$ the domain of $f'_j$ is a subset of the domain of $f_j$ and $f_j'$ is the corresponding restriction of $f_j$. For $r=0$ the restriction of $(f_1;n_1)$ to $X'$ is defined to be $(f_1|_{X'}; n_1)$.

Let $r\ge 1$. Then we set $f_1'=f_1|_{X'}$ and for all $x\in M^{m_1}$, if for some (necessarily unique) $y\in f_1(X')$ 
$$\big(f_2(x,-),\dots, f_{r+1}(x,-);m_2,\dots, m_{r}, n_2,\dots, n_{r+1}\big)$$
is an $(r-1)$-step fibering of $f_1^{-1}(y)\times f_1(X_1)$ over $C$, then we require 
$$\big(f'_2(x,-),\dots, f'_{r+1}(x,-);m_2,\dots, m_{r}, n_2,\dots, n_{r+1}\big)$$
to be the restriction of this fibering to $\big(f_1^{-1}(y)\cap X'\big)\times f_1(X')$, and if there is no such $y$, then the domains of
$ f'_2(x,-),\dots, f'_{r+1}(x,-)$ are empty. 

 \medskip\noindent
{\em Combining several $r$-step fiberings into one}. Fix distinct elements $0,1\in C$ (not to be confused with $0,1\in A$).  Suppose 
$$(f_1,\dots, f_r, f_{r+1}; m_1,\dots,m_{r}, n_1,\dots, n_r,n_{r+1})$$ with $f_j: X_j\to M^{n_j}$ ($j=1,\dots,r+1$) is an $r$-step fibering
of $X$ over $C$. This remains true when keeping $m$ and $X\subseteq M^m$ fixed but increasing any of $m_1,\dots, m_{r}, n_1,\dots, n_{r+1}$, while construing
$M^{m+n_1+\cdots + n_{j-1}}$ as $$M^m\times M^{n_1}\times \cdots \times M^{n_{j-1}},$$ and identifying any $M^{\mu}$, $\mu\in \{ m_1,\dots, m_r, n_1,\dots,n_{r+1}\}$,
with a subset of $M^{\mu+d}$ via $x\mapsto (x,\epsilon_1,\dots, \epsilon_d)$, for any $d$ and  $\epsilon_1,\dots, \epsilon_d\in \{0,1\}\subseteq C$. 

Next, let $X^0, X^1\subseteq M^m$ be definable and for $l=0,1$, let 
$$\big(f^l_1,\dots, f^l_{r+1};\vec{m}^l, \vec{n}^l\big),\quad \vec{m}^l=(m^l_1,\dots, m^l_r),\ \vec{n}^l=(n^l_1,\dots, n^l_{r+1})$$
be an $r$-step fibering of $X^l$ over $C$. Then we obtain an $r$-step fibering 
$$\big(f_1,\dots, f_{r+1};\vec{m}, \vec{n}\big), \qquad \vec{m}=(m_1,\dots, m_r),\  \vec{n}=(n_1,\dots, n_{r+1})$$
of $X:=X^0\cup X^1$ over $C$ as follows.

First, replacing $X^1$ by $X^1\setminus X^0$ and $\big(f^1_1,\dots, f^1_{r+1};\vec{m}^1, \vec{n}^1\big)$ by its restriction to $X^1\setminus X^0$ we arrange that $X^0$ and $X^1$ are disjoint. 

Secondly,  increasing any $m^l_j$ as indicated above, we arrange $\vec{m}^0=\vec{m}^1$.  
For $2\le j\le r$, set $m_j:= m^0_j=m^1_j$, and if $r\ge 1$, set
$m_1:=1+m^0_1=1+m^1_1$. 

By increasing the $n^l_j$ we also arrange $\vec{n}^0=\vec{n}^1$, and set $\vec{n}:=\vec{n}^0=\vec{n}^1$. Using suitable identifications of $M^{n^0_1}$ and $M^{n^1_1}$ with subsets of $M^{n_1}$
we arrange also that the sets $f^0_1(X^0)$ and $f^1_1(X^1)$ (subsets of $M^{n_1}$) are disjoint. 
Take $f_1$ such that its graph is the union of the graphs of $f^0_1$ and $f^1_1$. For $r=0$ this yields the $0$-step fibering 
$(f_1;n_1)$ of $X$ over $C$. Assume $r\ge 1$ and for $j=2,\dots, r+1$, let $f^0_j$ and $f^1_j$ have domain 
$X^0_j\subseteq M^{m_1-1}\times \cdots$ and $ X^1_j\subseteq M^{m_1-1}\times\cdots$, and define $f_j$ to have
domain $\big(\{0\}\times X^0_j\big)\cup \big(\{1\}\times X^1_j\big)\subseteq M^{m_1}\times \cdots$, with
$f_j(0,x,-)=f^0_j(x,-)$ and $f_j(1,x,-)=f^1_j(x,-)$ for $x\in M^{m_1-1}$.  Then $\big(f_1,\dots, f_{r+1};\vec{m}, \vec{n}\big)$
is an $r$-step fibering of $X$ over $C$, as promised. 
 
  Thus, given $N\in \N^{\ge 1}$, definable $X^1,\dots, X^N\subseteq M^m$, and for each of $X^1,\dots, X^N$ an $r$-step fibering over $C$ 
we can combine these fiberings into a single $r$-step fibering of $X^1\cup\cdots\cup X^N$ over $C$.  This device will be used in the proof of the next lemma.
%~\ref{l5}.
%, but before stating it we now take advantage of having the above lemma available without assuming $\cM$ is $\omega$-saturated. 

 %\bigskip\noindent
%We now restore the requirement that $\cM$ is $\omega$-saturated. In this setting we have
%for definable :

\medskip\noindent

\begin{lemma}~\label{l5} Assume $\cM$ is $\omega$-saturated. Let  $X\subseteq M^m$ be definable. Then
$$ X \text{ has an $r$-step fibering over $C$}\ \Longleftrightarrow\ X\in \cC_r.$$
\end{lemma}
\begin{proof} An obvious induction on $r$ yields $\Rightarrow$.  
For $\Leftarrow$ we also use induction on $r$. The case $r=0$ is trivial. Let $r\ge 1$, $X\in \cC_{r}$, and take a definable 
$f: X\to M^n$ such that $f(X)\in \cC_{r-1}$, and $f^{-1}(y)\in \cC_{r-1}$ for all $y\in f(X)$.  Augmenting $\cal{L}$ by names for finitely many elements of $M$ and expanding $\cM$ accordingly we arrange that $f$ and $X$ are $0$-definable. Assume inductively that for all $y\in f(X)$ the set $f^{-1}(y)\times f(X)\subseteq M^{m+n}$ has an $(r-1)$-step fibering over $C$. In other words, for every $y\in f(X)$ there is a tuple 
$$\big(f_2,\dots, f_{r+1}; m_2,\dots, m_r, n_2,\dots, n_{r+1}\big)$$
such that for $j=2,\dots, r+1$: $f_j$ is a $0$-definable map $X_j\to M^{n_j}$, with $0$-definable
$$X_j\ \subseteq\ M^k\times M^{m_2}\times \cdots \times M^{m_{j-1}}\times M^{m+n+(n_2+\cdots+ n_{j-1})}$$
and there is a parameter $p\in M^k$ for which $X_2(p)=f^{-1}(y)\times f(X)\subseteq M^{m+n}$ and
$$\big(f_2(p,-),\dots, f_{r+1}(p,-); m_2,\dots, m_r, n_2,\dots, n_{r+1}\big)$$
is an $(r-1)$-step fibering of $f^{-1}(y)\times f(X)$ over $C$. By $\omega$-saturation and  ``uniform definability" this yields a finite set $E$ of such tuples such that 
 for all $y\in f(X)$ there is a tuple as above in $E$ and a $p\in M^k$ for which $X_2(p)=f^{-1}(y)\times f(X)$ and
$$\big(f_2(p,-),\dots, f_{r+1}(p,-); m_2,\dots, m_r, n_2,\dots, n_{r+1}\big)$$
is an $(r-1)$-step fibering of $f^{-1}(y)\times f(X)$ over $C$.  Using the earlier remark about combining fiberings
we obtain from the  tuples in $E$ 
a single such tuple 
$$\big(f_2,\dots, f_{r+1}; m_2,\dots, m_r, n_2,\dots, n_{r+1}\big)$$
(not necessarily in $E$) that works for all $y\in f(X)$.   Then  the tuple
$$\big(f, f_2,\dots, f_{r+1}; k, m_2,\dots, m_r, n, n_2,\dots, n_{r+1}\big)$$
is an $r$-step fibering of $X$ over $C$.
\end{proof} 

\noindent
Let $X\subseteq M^m$ be definable. We declare $X$ to be $r$-step fiberable over $C$ if it has an $r$-step fibering over $C$.  For $\omega$-saturated $\cM$ this agrees with the previously defined concept, in view of Lemma~\ref{l5}.  By``uniform definability" we have for any elementary extension
$\cM^*=(M^*,\dots)$ of $\cM$: $X$ is $r$-step fiberable over $C$ with respect to $\cM$ as ambient structure iff $X^*$ is $r$-step fiberable over $C^*$ with respect to $\cM^*$ as ambient structure.

\medskip\noindent
We return to the situation of the previous subsection: $\cM$ is $\omega$-saturated, $C$ is stably embedded, and (IA1) and (IA2) hold for a given $r$.  By Proposition~\ref{prop+}, (IA1) is inherited in going from $r$ to $r+1$. By the next result this is also the case for (IA2).

\begin{cor} Let $\cal{S}\subseteq M^k\times M^m$ be definable such that $\cal{S}(p)\in \cC_{r+1}$ for all $p\in M^k$. Then the set
$\{\mu\big(\cal{S}(p)\big):\ p\in M^k\}$ is finite, and for every $a$ the set $\{p\in M^k:\ \mu\big(\cal{S}(p)\big)=a\}$ is definable . 
\end{cor} 
\begin{proof} For every $p\in M^k$ the set $\cal{S}(p)$ has by Lemma~\ref{l5} an $(r+1)$-step fibering over $C$. As in the proof of Lemma~\ref{l5}, this yields a finite set $E$ of tuples 
$$\big( f_1,\dots, f_{r+1}, f_{r+2}; m_1,\dots,m_{r+1}, n_1,\dots, n_{r+1},n_{r+2}\big)$$
where for $j=1,\dots, r+2$: $f_j$ is a definable map $ \Sigma_j\to M^{n_i}$ with (definable) domain 
$\Sigma_j\subseteq M^l\times M^{m_1}\times\cdots\times M^{m_{j-1}}\times M^{m+n_1+\cdots + n_{j-1}}$, and for each $p\in M^k$ there is 
a tuple in $E$ as displayed and $q\in M^{l}$
such that $\cal{S}(p)=\Sigma_1(q)$ and the tuple
$$ \big(f_1(q,-),\dots, f_{r+1}(q,-), f_{r+2}(q,-); m_1,\dots, m_{r+1}, n_1,\dots, n_{r+1}, n_{r+2}\big)$$
is an $(r+1)$-step fibering of the domain $\Sigma_1(q)\subseteq M^m$ of $f_1(q,-)$ over $C$.  

Consider one such tuple $\tau\in E$ as displayed above. Let $Q$ be the definable
set of $q\in M^{l}$ such that
for some $p\in M^k$ we have $\cal{S}(p)=\Sigma_1(q)$ and 
$$ \big(f_1(q,-),\dots, f_{r+1}(q,-), f_{r+2}(q,-); m_1,\dots, m_{r+1}, n_1,\dots, n_{r+1},n_{r+2}\big)$$
is an $(r+1)$-step fibering of the domain $\Sigma_1(q)\subseteq M^m$ of $f_1(q,-)$ over $C$.
Each $q\in Q$ yields a map
$f_1(q,-): \Sigma_1(q)\to M^{n_1}$ such that its preimages of points $y\in M^{n_1}$ are in $\cC_r$. The maps $f_1(q,-)$ with $q\in Q$ make up a definable family $\cal{F}_\tau$ of maps from subsets of $M^m$ into $M^{n_1}$ as described just before Lemma~\ref{l2+}. That lemma then tells us that $\{\mu\big(\Sigma_1(q)\big): q\in Q\}$ is finite, and that
$\{q\in Q:\ \mu\big(\Sigma_1(q)\big)=a\}$ is definable for every $a$. This holds for all $\tau\in E$, which yields the
desired result.
\end{proof} 

\noindent
This finishes the inductive proof of Theorem~\ref{thm} in the case of $\omega$-saturated 
$\cM$.

\subsection*{Reduction to the $\omega$-saturated case}  Here $C$ is stably embedded in $\cM$, but we do not assume $\cM$ is $\omega$-saturated. 

Let $\cM^*$ be an elementary extension of $\cM$. We define an $A$-valued Fubini measure $\mu^*$ on  the full subcategory 
$\Def(C^*)$ of $\Def(\cM^*)$: Let $Y\subseteq (C^*)^n$ be definable in $\cM^*$; take an $\cal{L}$-formula $\phi(x,y)$ with $x=(x_1,\dots, x_m)$, $y=(y_1,\dots, y_n)$, and a point $p_Y\in (C^*)^m$ such that $Y=\{q\in (C^*)^n: \cM^*\models \phi(p_Y,q)\}$. To simplify notation, set $\phi(p, C^n):=\{q\in C^n:\ \cM\models \phi(p,q)\}$ for $p\in C^m$, and likewise define
$\phi\big(p, (C^*)^n\big)\subseteq (C^*)^n$ for $p\in (C^*)^m$; in particular, $Y=\phi\big(p_Y, (C^*)^n\big)$. 
Now $C^m$ is the disjoint union of the definable sets $X_a:=\{p\in C^m:\ \mu\big(\phi(p,C^m)\big)=a\}$ (which is empty for all but finitely many $a$). 
Set $\mu^*(Y):=a$, where $a$ is unique such that $p_Y\in (X_a)^*$. It is routine to verify that this particular $a$ does not depend on the choice of $\phi$ and that it yields indeed a Fubini measure on $\Def(C^*)$. 

Now take such an elementary extension $\cM^*$ to be $\omega$-saturated. Then $\mu^*$ extends uniquely to an $A$-valued Fubini measure $\mu^*$ on the full subcategory $\Def(C^*)^{\f}$ of $\Def(\cM^*)$. 
For $X\in \Def(C)^{\f}$ we have $X^*\in \Def(C^*)^{\f}$; set $\mu(X):= \mu^*(X^*)$.
This extends the original $A$-valued Fubini measure $\mu$ on $\Def(C)$ to an $A$-valued Fubini measure $\mu$ on $\Def(C)^{\f}$.  As to uniqueness, suppose $X$ is fiberable over $C$, and take an $r$-step fibering $(f_1,\dots, f_r, f_{r+1}; m_1,\dots, m_{r}, n_1,\dots, n_{r+1})$
of $X$ over $C$. 
For $r=0$ the definable map $f_1: X \to C^{n_1}$ is injective, and so by Fubini
$\mu(X)$ is uniquely determined by $\mu: \Def(C)\to A$. For $r\ge 1$ the definable map $f_1: X\to M^{n_1}$ is such that
$f_1(X)$ and every
$f_1^{-1}(y)$ with $y\in f_1(X)$ is $(r-1)$-step fiberable over $C$, and we can assume inductively that the values of 
$\mu\big(f(X)\big)$ and $\mu\big(f_1^{-1}(y)\big)$ for
$y\in f_1(X)$ are uniquely determined by $\mu: \Def(C)\to A$, which again by Fubini uniquely deternines $\mu(X)$. 
This concludes the proof of Theorem~\ref{thm}.

\section{Fiberability and Co-analyzability}\label{co}

\noindent
Fiberability is a somewhat ad hoc notion, but agrees under natural model-theoretic conditions with the more intrinsic notion of co-analyzability from \cite{HHM}, as we shall see. {\it Throughout this section $X\subseteq M^m$ is definable}.  

\subsection*{Fiberability implies co-analyzability}
We recall here the recursion on $r$ that defines ``$X$ is co-analyzable in $r$ steps" (tacitly: relative to $C$ in $\cM$), where $\cM$ is assumed to be $\omega$-saturated: \begin{enumerate}
\item  $X$ is co-analyzable in $0$ steps iff $X$ is finite;
\item  $X$ is co-analyzable in $r+1$ steps iff for some definable set $R \subseteq M^m\times C$ we have: $\breve{R}(y)$ is co-analyzable in $r$ steps for every $y\in C$, and 
$X:= \bigcup_{y\in C}\breve{R}(y)$. 
\end{enumerate} 
 See \cite[p. 35]{ADH} for how this extends to not necessarily $\omega$-saturated $\cM$ in such a way that
 for any elementary extension $\cM^*$ of $\cM$: $X$ is co-analyzable in $r$ steps relative to $C$ in $\cM$ iff $X^*$ is co-analyzable in $r$ steps relative to $C^*$ in $\cM^*$.  Call $X$ {\em co-analyzable\/} if $X$ is co-analyzable in $r$ steps for some $r$. 
 
 By \cite{HHM} this notion can be characterized as follows:
 
 \begin{prop}\label{hhm} 
Let $T$ be a complete $L$-theory, where $L$ is a countable one-sorted language. Let $y$ be a single variable and $D(y)$ an $L$-formula that defines in each model $\cal{N}=(N;\cdots)$ of $T$ a set $D(N)$ with $|D(N)|\ge 2$.
Then the following conditions on an $L$-formula $\varphi(x)$ with $x=(x_1,\dots,x_m)$ are equivalent:
\begin{enumerate}
\item[\textup{(i)}] for some  $\cal{N}\models T$, $\varphi(N^m)$ is co-analyzable relative to $D(N)$;
\item[\textup{(ii)}] for all $\cal{N}\models T$, $\varphi(N^m)$ is co-analyzable relative to $D(N)$;
\item[\textup{(iii)}] for all  $\cal{N}\models T$, if 
$D(N)$ is countable, then so is $\varphi(N^m)$;
\item[\textup{(iv)}] for all $\cal{N}\preceq\cal{N}^*\models T$, if
$D(N)=D(N^*)$, then $\varphi(N^m)=\varphi((N^*)^m)$.
\end{enumerate}
\end{prop}

\noindent
The equivalences (i)$\Leftrightarrow$(ii)$\Leftrightarrow$(iv) hold without assuming that $L$ is countable:
For (i)$\Leftrightarrow$(ii)$\Rightarrow$(iv), reduce to countable $L$ by taking suitable reducts. For (iv)$\Rightarrow$(ii), use also that (iv) amounts to $T_2\models \sigma$ where $T_2$ is the theory of elementary pairs of models of $T$ and $\sigma$ is a certain sentence, depending on $\varphi$, in the language of $T_2$. Note that in addition we can restrict $\cal{N}$ and $\cal{N}^*$ in (iv) to be  $\omega$-saturated.

\begin{lemma}\label{fico} If $X$ is fiberable over $C$, then  $X$ is co-analyzable relative to $C$. 
\end{lemma} 
\begin{proof} 
We arrange that $\cM$ is $\omega$-saturated. 
Below we prove by induction on $r$:  if $X$ is $r$-step fiberable over $C$, then $X$ is co-analyzable relative to $C$. 
For $r=0$, if $f: X\to C^n$ is definable and injective, then $X$ is co-analyzable in $n$ steps relative to $C$, by an easy induction on $n$. Thus the claimed implication holds for $r=0$. Assume this implication holds for a certain $r$, and let
$X$ be $(r+1)$-step fiberable over $C$.
To show $X$ is co-analyzable relative to $C$ in $\cM$ we 
augment $\cal{L}$ by names for finitely many elements of $M$ to arrange that $X$ is $0$-definable, so $X=\varphi(M^m)$ with $\varphi(x)$ an $\cal{L}$-formula, $x=(x_1,\dots, x_m)$. We now apply (iv)$\Rightarrow$(ii) of Proposition~\ref{hhm} without assuming $L=\cal{L}$ is countable, with $T:=\Th(\cM)$ and $D(y)$ an $\cal{L}$-formula such that $D(M)=C$. 
So let $\cal{N}\preceq \cal{N}^*\models T$ with $D(N)=D(N^*)$; it suffices to show that then $\varphi(N^m)=\varphi\big((N^*)^m\big)$.
We also assume $\cal{N}$ is $\omega$-saturated, as we may by a remark following Proposition~\ref{hhm}. Since $\cal{M}\equiv \cal{N}$, the set $\varphi(N^m)$ is $(r+1)$-step fiberable over $C^{\cal{N}}:= D(N)$ in $\cal{N}$. Take $f: \varphi(N^m)\to Y$ where $Y\subseteq N^n$ and $f$ are definable in $\cal{N}$, such that $Y$ is $r$-step fiberable over $C^{\cal{N}}$ and for all
$y\in Y$ the fiber $f^{-1}(y)$ is $r$-step fiberable over $C^{\cal{N}}$. Thus by our inductive assumption and Proposition~\ref{hhm}, $Y$ and the fibers
$f^{-1}(y)$ for $y\in Y$ do not change in passing from $\cal{N}$ to $\cal{N}^*$. Hence the same is true for $\varphi(N^m)$, that is,
$\varphi(N^m)=\varphi\big((N^*)^m\big)$. 
\end{proof}

\subsection*{The o-minimal and strongly minimal cases}
{\em In this subsection $C$ is stably embedded in $\cM$}. 
 By $(C;\cM)$ we denote $C$ equipped with the structure induced by $\cM$: its underlying set is $C$ and its $0$-definable sets are the $Y\subseteq C^n$ that are $0$-definable in $\cM$. (The actual language for which $(C;\cM)$ is a structure in the sense of model theory doesn't matter for us.) Note that then for any $Y\subseteq C^n$: 
 $$Y \text{  is definable in }\cM\ \Longleftrightarrow\  Y \text{ is definable in }(C;\cM).$$  
 If $(C,\cM)$ is an o-minimal expansion of a divisible ordered abelian group with respect to a group operation on $C$ and a total ordering on $C$ that are definable in $\cM$,
 then by the proof of \cite[Lemma 6.3]{ADH},
$$X \text{ is fiberable over }C\ \Longleftrightarrow\ X \text{ is co-analyzable relative to }C.$$
Suppose next that $(C;\cM)$ is strongly minimal. This includes $C$ being infinite, and so we have an infinite sequence $c_0, c_1, c_2,\dots$
of distinct elements of $C$. We use this sequence to assign to any $Y\subseteq C$, definable in $(C;\cM)$, a finite subset $Y_0$:  if $Y$ is finite, then $Y_0:=Y$, and if $Y$ is cofinite, then $Y_0:=\{c_k\}$ where $k$ is minimal with $c_k\in Y$. Note that if $Y$ is nonempty, then so is $Y_0$.

For definable $\cal{S}\subseteq  M^m\times C$, take $\cal{S}_0\subseteq M^m\times C$ such that $\cal{S}_0(x)=\cal{S}(x)_0$ for $x\in M^m$, and note that $\cal{S}_0$ is also definable.

\begin{lemma}\label{cofi} Suppose $(C;\cM)$  is strongly minimal and has $\rm{EI}$.  Then:
$$ X \text{ is co-analyzable in $r$ steps relative to $C$}\ \Longrightarrow\ X \text{ is $r$-step fiberable over $C$}.$$
\end{lemma}
\begin{proof} Clear for $r=0$. Assume it holds for $r$, and let $X$ be co-analyzable in $r+1$ steps relative to $C$. We also arrange $\cM$ is $\omega$-saturated. 
Take definable $\cal{S}\subseteq M^m\times C$ such that $\breve{\cal{S}}(y)$ is $r$-step fiberable
over $C$ for all $y\in C$ and $X=\breve{\cal{S}}(C)$. 
For $x\in X$ we have the nonempty definable set $\cal{S}(x)\subseteq C$, so replacing $\cal{S}$ by its subset $\cal{S}_0$ we arrange that $\cal{S}(x)$ is finite for all $x\in X$. Since $C$ is stably embedded in $\cM$ this yields $e\in \N^{\ge 1}$ such that $1\le |\cal{S}(x)|\le e$ for all $x\in X$. 
 Now $(C;\cM)$ having EI gives an encoding of nonempty finite subsets of $C$ of size $\le e$ by elements of a fixed power $C^n$,
 and hence a definable map $f: X\to C^n$ such that for all $x,x'\in X$: $\cal{S}(x)= \cal{S}(x')\Leftrightarrow f(x)= f(x')$. Thus for $x\in X$ and $f(x)=z$ we have
 $$f^{-1}(z)\ \subseteq\ \bigcup_{y\in \cal{S}(x)} \breve{\cal{S}}(y),$$
 so $f^{-1}(z)\in \cC_r$. Thus $f$ witnesses that $X\in \cC_{r+1}$.
\end{proof} 

\subsection*{Connection to internality} Fiberability of $X$ over $C$ and co-analyzability of $X$ relative to $C$ are two
ways in which $C$ can control $X$. Here are other ways: 
\begin{enumerate}
\item $X$ is internal to $C$, that is, $X=f(Y)$ for some definable $Y\subseteq C^n$ and definable $f: Y\to M^m$. (For nonempty $X$ one can take $Y=C^n$.) 
\item $X$ is almost internal to $C$, that is, for some definable $\cal{S}\subseteq C^n\times M^m$ and $e$ we have $|\cal{S}(y)|\le e$ for all $y\in C^n$ and $\cal{S}(C^n)=X$. 
\end{enumerate} 
It is easy to check that we have the following implications: 
\begin{align*} X \text{ is internal to }C\ &\Longrightarrow\ X \text{ is almost internal to }C\\
X \text{ is almost internal to }C\ &\Longrightarrow\ X\text{ is co-analyzable relative to }C.
\end{align*} 
In discussing various special cases in the next section we show these implications cannot be reversed in general. They can be reversed under certain conditions. For example, if there is a definable total ordering on $M$, then clearly:
$$X \text{ is internal to }C\ \Longleftrightarrow\  X \text{ is almost internal to }C.$$

\begin{lemma} Suppose $C$ is stably embedded in $\cM$, $(C;\cM)$ has ${\rm{EI}}$, and  $X$ is internal to $C$. Then $X\in \cC_0$.
\end{lemma}
\begin{proof} Take definable $f: Y\to M^m$ with $Y\subseteq C^\nu$, $\nu\in \N$, such that $X=f(Y)$. Since $(C;\cM)$ has EI, we have a definable map $g:Y\to C^n $ such that for all $y,z\in Y$ we have $f(y)=f(z)\Leftrightarrow g(y)=g(z)$. This yields a definable bijection $$\qquad\qquad \qquad \qquad f(y)\mapsto g(y)\ :\  X\to g(Y),\qquad (y\in Y). \qquad \qquad \qquad \qquad \qedhere$$
\end{proof} 

\subsection*{Analyzability} 
The concept of {\em analyzability\/} is more popular in model theory than co-analyzability, but is usually defined for types
in the environment of a monster model, instead of definable sets. I have been told that for definable sets and $\omega$-saturated $\cM$ its recursive definition should be as follows: 

{\em $X$ is $0$-step analyzable in $C$ iff $X$ is internal to $C$, and $X$ is $(r+1)$-step analyzable in $C$ iff there is a $\Def(\cM)$-morphism $f: X \to Y$ such that $Y$ is $r$-step analyzable over $C$ and all fibers $f^{-1}(y)$, $y\in Y$, are internal to $C$. Finally, $X$ is analyzable in $C$ iff $X$ is $r$-step analyzable in $C$ for some $r$}. 

This extends to not necessarily $\omega$-saturated $\cM$ by requiring in the $(r+1)$-clause the existence of a definable 
family of maps from subsets of $C^n$ (for a fixed $n$) into $M^m$ such that for each $y\in Y$ some member of the family has image $f^{-1}(y)$. 
With this extended notion and any elementary extension $\cM^*$ of $\cM$, 
$$X \text{ is $r$-step analyzable in }C\ \Longleftrightarrow\  X^* \text{ is $r$-step analyzable in }C^*.$$ 
Using this it follows from Proposition~\ref{hhm} that if $X$ is analyzable in $C$, then $X$ is co-analyzable relative to $C$. Ben Castle and Rahim Moosa have assured me that the converse should also hold under some natural model-theoretic conditions to be satisfied in particular by differentially closed fields (and probably by $\T$) with $C$ its constant field. I have not pursued this, but it would be good to have this confirmed. One motivation for this: it might help in showing that for $\T$ and for differentially closed fields the notions of  fiberability over $C$ and co-analyzability relative to $C$ do not collapse for any $r$ to $r$-step fiberability over $C$ and co-analyzability in $r$-steps relative to $C$, respectively. This is plausible in view of \cite{J} where Jin constructs in a ``big" differentially closed field for any $r\ge 1$ a $1$-type that is $r$-step analyzable but not $(r-1)$-step analyzable (in the constant field).

\section{Application to $\T$ and to differentially closed fields} \label{special}

\subsection*{The case of $\T$} We construe here $\T$ as an ordered differential field. Its subfield $\R$ is its constant field, and
a set $Y\subseteq \R^n$ is definable in $\T$ iff $X$ is semialgebraic in the sense of $\R$. In particular, $\R$ is stably embedded in
$\T$. Thus by Theorem~\ref{thm}, the $\N_{\trop}$-valued Fubini measure $Y\mapsto \dim Y$ on $\Def(\R)$ extends uniquely to an $\N_{\trop}$-valued Fubini measure on the full subcategory $\Def(\R)^{\f}$ of $\Def(\T)$. Likewise for the $\Z$-valued Fubini measure $Y\mapsto E(Y)$ on $\Def(\R)$. 

All this holds also for any $\upo$-free, newtonian, Liouville closed $H$-field $K$ and its constant field instead of
$\T$ and its constant field $\R$; this includes all ordered differential fields $K \equiv\T$.  The extended Euler characteristic now yields an answer to a test question raised in \cite[p. 34, end of Section 5]{ADH}:

\begin{cor} With $K$ as above, let $X\subseteq K$ be definable in $K$. Then $X$ with the ordering induced by $K$
is not elementarily equivalent to the ordered set $\omega$. 
\end{cor}
\begin{proof} Suppose $X\equiv \omega$. Then $X$ is discrete as a subspace of $K$ with its order topology. Hence $X$ is fiberable
over $C$ by \cite{ADH}. Now $X\equiv \omega$ also yields  a definable bijection between $X$ and $X\setminus \{\text{the least element of $X$}\}$, but the extended Euler characteristic then gives $E(X)=E(X)-1$, a contradiction. 
\end{proof} 

\noindent
%The main point of this corollary is that we have no other proof of it. 
Let $X:=\{y\in \T:\ yy''=(y')^2\}$. Then $X$ is $1$-step fiberable over $\R$, witnessed by
 $f: X\to \R$ given by $f(y)=y'/y$ for $y\ne 0$ and $f(0)=0$: for $0\ne z\in \R$ we have $y\in X$ with $f(y)=z$, and then
 $f^{-1}(z)=y\R^\times$, and
 $f^{-1}(0)=\R$. Thus for the extended semialgebraic dimension and Euler characteristic we have 
 $$\dim X\ =\ 2, \qquad E(X)\ =\ 3.$$
 In \cite[pp. 31, 32]{ADH} we showed that $X$ is not internal to $\R$, and so $X$ is in particular not definably isomorphic to any set in $\Def(\R)$.
Any set in $\Def(\R)$ is determined up-to-definable-isomorphism by its dimension and Euler characteristic, by \cite[p. 132]{D}, but the above $X$ shows this is no longer the case for $\Def(\R)^{\f}$.

\subsection*{Tame pairs of real closed fields} In this subsection $R$ is a proper real closed field extension of $\R$. Then
 $Y\subseteq \R^n$ is definable in $\cM=(R,\R)$ iff $Y$ is real semialgebraic, that is, semialgebraic  in the sense of the real 
closed field $\R$. In particular, $\R$ is stably embedded in $\cM$.  The following equivalences 
for definable $X\subseteq R^m$ are from \cite{AD}: 
\begin{align*}  &X \text{ is definably isomorphic to a real semialgebraic set}\\
&\ \Longleftrightarrow\  X  \text{ is fiberable over }\R\  \Longleftrightarrow\ X \text{ is discrete}.
\end{align*} 
``Discrete" is with respect to the order topology on $R$ and the corresponding product topology on $R^n$. Thus here
$\Def(\R)^{\f}=\Def(\R)^{\iso}$, and so  the dimension and Euler characteristic 
on $\Def(\R)$ extend uniquely to Fubini measures on $\Def(\R)^{\f}$, without requiring Theorem~\ref{thm}. 

\medskip\noindent
We recall that the models of $\operatorname{Th}(R,\R)$ are the pairs $(K,\k)$ with $K$ a real closed field and $\k$ a proper real closed subfield such that
$\mathcal{O}=\k+\smallo$ where $\mathcal{O}$ is the convex hull of $\k$ in $K$ and $\smallo$ is the maximal ideal of the valuation ring $\mathcal{O}$.  

\subsection*{Algebraically closed fields}  Let $\k$ be an algebraically closed field. The objects of $\Def(\k)$ are the subsets of the affine spaces $\k^n$ that are constructible with respect to the Zariski topology on these spaces (Chevalley-Tarski). 

\begin{prop}\label{uniq}  If $\operatorname{char}(\k)>0$, then there is no $\Z$-valued Fubini measure on $\Def(\k)$. If 
$\operatorname{char}(\k)=0$, then there is a unique $\Z$-valued Fubini measure on $\Def(\k)$. 
\end{prop} 
\begin{proof} Suppose $\mu$ is a $\Z$-valued Fubini measure on $\Def(\k)$. Then $\mu(\k)=1$: 
 if $\operatorname{char}(\k)\ne 2$,  the fibers of 
 $y\mapsto y^2: \k\to \k$
have size $2$ except at $0$, so 
$$\mu(\k)\  =\ 1+2\big(\mu(\k)-1\big),$$
 and thus $\mu(\k)=1$; if $\operatorname{char}(\k)=2$, the map $y\mapsto y^3: \k\to \k$ gives likewise $\mu(\k)=1+3\big(\mu(\k)-1\big)$, again resulting in $\mu(\k)=1$. 

Now assume $\operatorname{char}(\k)=p>0$, and set $X:=\{(x,y)\in \k^2: x^p+x=y\}$. The map $(x,y)\mapsto x: X\to \k$ is bijective, so
$\mu(X)=\mu(\k)=1$, and the fibers of the map $(x,y)\to y: X\to \k$ have size $p$, so $\mu(X)=p\mu(\k)=p$, a contradiction.

Finally, suppose $\operatorname{char}(\k)=0$. Then $\k$ has a real closed subfield $\k_0$ such that
$\k=\k_0[\imag]$ with $\imag^2=-1$. Identify $\k$ with $\k_0^2$ via $x+y\imag\mapsto (x,y)$ for $x,y\in \k_0$.
Then the constructible sets $Y\subseteq \k^n$ become semialgebraic sets $Y\subseteq \k_0^{2n}$, and as such have an
Euler characteristic $E(Y)\in \Z$. As a function of $Y\in \Def(\k)$, this is a Fubini measure on $\Def(\k)$. If $Y\subseteq \k$ is constructible, then $Y$ is finite, or cofinite in $\k$, so $\mu(Y)=E(Y)$. Now Lemma~\ref{munu} yields for all $n$ and constructible  $Y\subseteq \k^n$ that 
$\mu(Y)=E(Y)$, and thus uniqueness as claimed.  
\end{proof} 

\noindent
Let now $\k$ be a proper subfield of the algebraically closed field $K$, and $\cM:=(K, \k)$. Then
$\cM$ is $\omega$-stable with $\text{MR}(\k)=1$, $\text{MR}(K)=\omega$, and a set $Y\subseteq \k^n$ is definable in $\cM$ iff $Y$ is constructible in the sense of the algebraically closed field $\k$. Thus $\k$ is stably embedded in $\cM$.   
Let $X\subseteq K^m$ be definable. Then
%\begin{align*} 
$$X\text{ is almost internal to }\k\
\Leftrightarrow\  X \text{ is co-analyzable relative to }\k\
  \Leftrightarrow\ \operatorname{MR}(X)<\omega.$$
  % \text{ has finite Morley rank}.
%\end{align*} 
These equivalences are in \cite{AD}, which also has an example of a definable $X\subseteq K$ that is almost internal to $\k$ but not internal to $\k$. Another proof of the first equivalence using \cite{P} was pointed out to me by Pillay.   Lemmas~\ref{fico} and ~\ref{cofi} also give: 
$$X \text{ is fiberable over }\k\ \Longleftrightarrow\  X \text{ is co-analyzable relative to }\k.$$ 
 By Lemma~\ref{smfub} the function $Y\mapsto \text{MR}(Y)$
on $\Def(\k)$ is a Fubini measure with values in $\N_{\trop}$. By Theorem~\ref{thm} it extends uniquely to
an $\N_{\trop}$-valued Fubini measures on $\Def(\k)^{\f}$; this extension is necessarily
$X\mapsto \text{MR}(X)$: the above relation to almost internality implies that it is a Fubini measure on $\Def(\k)^{\f}$. 

\medskip\noindent
Now assume $\operatorname{char}(\k)=0$. Then we have the unique $\Z$-valued Fubini measure $Y\mapsto E(Y)$ on $\Def(\k)$ from the proof of Proposition~\ref{uniq}. By Theorem~\ref{thm} it extends uniquely to a $\Z$-valued Fubini measure on 
$\Def(\k)^{\f}$. 
Although "fiberable over $\k$" collapses here to ``almost internal to $\k$", we still need all of Section~\ref{prelim} and much of Section~\ref{proof}  to obtain this extension result.

\subsection*{The case of differentially closed fields} Let $K$ be a differentially closed field, that is, a model of $\text{DCF}_0$. 
Let $C$ be its constant field. Then $Y\subseteq C^n$ is definable in $K$ iff $Y$ is definable in the algebraically closed field $C$, in particular, $C$ is stably embedded in $K$. Recall also that $K$ is $\omega$-stable, with $\text{MR}(C)=1$ and
$\text{MR}(K)=\omega$. Working in $\cM:=K$ we have for definable $X\subseteq K^m$:
$$ X \text{ is fiberable over }C\ \Longleftrightarrow\ 
X \text{ is co-analyzable relative to }C\ \Longrightarrow\ \operatorname{MR}(X)<\omega. $$
The equivalence here follows from Lemmas~\ref{fico} and  ~\ref{cofi}. The implication is obtained from the more precise fact that if $X$ is co-analyzable in $r$-steps relative to $C$, then $\text{MR}(X)\le r$ (by induction on $r$). 

In the previous subsection we considered the Fubini measures $$Y\mapsto \text{MR}(Y), \qquad Y\mapsto E(Y)$$
on $\Def(\k)$, taking values in $\N_{\trop}$ and in $\Z$, respectively. Taking now $\k=C$, Theorem~\ref{thm} yields unique extensions of these functions to
$\N_{\trop}$-valued and $\Z$-valued Fubini measures on $\Def(C)^{\f}$.

\section{A many-sorted version}\label{ms}  

\noindent
So far we assumed $\cM$ to be one-sorted. But many structures are more naturally viewed as many-sorted; for example, valued fields as $3$-sorted structures with residue field and value group as extra sorts. It is worth noting that for a henselian valued field of equicharacteristic $0$, its
residue field and value group are stably embedded in the ambient 3-sorted valued field.

Let $\cM=\big((M_\sigma)_{\sigma\in \Sigma};\cdots\big)$ be a many-sorted structure, $\Sigma$ its set of sorts.
For every tuple $\vec \sigma=(\sigma_1,\dots, \sigma_m)\in \Sigma^m$ this yields the product set 
$\cM_{\vec \sigma}=M_{\sigma_1}\times \cdots\times M_{\sigma_m}$. Now the category $\Def(\cM)$ has as its objects the
definable sets $X\subseteq \cM_{\vec\sigma}$, with $\vec\sigma$ part of specifying $X$ as such an object;  for $X,Y\in \Def(\cM)$ the morphisms $X\to Y$ are the definable maps $X\to Y$. Let $\cC$ be a full subcategory of $\Def(\cM)$ satisfying
conditions (a), (b), (c) from the introduction, where instead of $M^m$, $M^n$, and $M^{m+n}$ we have $M_{\vec{\sigma}}$, 
$M_{\vec{\tau}}$ and $M_{\vec{\sigma}\vec{\tau}}$ with $\vec{\sigma}\in \Sigma^m$ and $\vec{\tau}\in \Sigma^n$; the concept
``$A$-valued Fubini measure on $\cC$" is defined as before, with $M_{\vec\sigma}$ instead of $M^m$. 

Let there be given a subset
$\Sigma_0$ of $\Sigma$ and for each $\sigma\in \Sigma_0$ a nonempty $0$-definable set $C_{\sigma}\subseteq M_{\sigma}$.
Then we define the full subcategory $\cC$ of $\Def(\cM)$ to have as its objects the definable subsets of the products 
$$C_{\vec{\tau}}\ :=\  C_{\tau_1}\times \cdots \times C_{\tau_n}\ \subseteq\ M_{\vec{\tau}}, \qquad \vec{\tau}\in \Sigma_0^n.$$
(This is our many-sorted analogue of the $0$-definability of $C$ in $\cM$.) Note that this $\cC$ satisfies conditions
(a), (b), (c) as amended above. To avoid trivial cases we assume further that some $C_{\sigma}$ with $\sigma\in \Sigma_0$ has more than one element. 
We then define $\cC$ to be {\em stably embedded in $\cM$} if for every $0$-definable $\cal{S}\subseteq M_{\vec\sigma}\times M_{\vec \tau}$ with $\vec\sigma\in \Sigma^m$, $\tau\in \Sigma_0^n$, there is a $0$-definable $\cal{T}\subseteq M_{\vec\rho}\times M_{\vec\tau}$, with $\vec\rho\in \Sigma_0^\mu$ for some $\mu\in \N$, such that for all $p\in M_{\vec\sigma}$ there exists $c\in C_{\vec \rho}$ with $\cal{S}(p)\cap C_{\vec\tau} = \cal{T}(c)\cap C_{\vec \tau}$. 

Let $\cM$ be $\omega$-saturated.  By recursion on $r\in \N$ we define for $X\in \Def(\cM)$ the concept
``$X$ is $r$-step fiberable over $\cC$":  this means the existence of a definable $f: X \to C_{\vec\tau}$, $\tau\in \Sigma_0^n$, such that 
if $r=0$, then $f$ is injective, and if $r\ge 1$, then $f(X)$ and every fiber $f^{-1}(y)$, $y\in C_{\vec\tau}$,
is $(r-1)$-step fiberable over $\cC$. 

For  $\cM$ that is not necessarily $\omega$-saturated and $X\in \Def(\cM)$ we extend the notion ``$X$ is $r$-step fiberable over $\cC$'' as in the subsection {\em Completing the inductive step}, in such a way that it is equivalent to
``$X^*$ is $r$-step fiberable over $\cC^*$'' where $\cM^*=\big((M_\sigma)_{\sigma\in \Sigma};\cdots\big)$ is any elementary extension of $\cM$.

Let $\cC^{\f}$ be the full subcategory of $\Def(\cM)$ whose objects are the $X\in \Def(\cM)$
that are $r$-step fiberable over $\cC$ for some $r$. Then $\cC^{\f}\supseteq \cC$ and $\cC^{\f}$ satisfies (a), (b), (c) as amended above.  Our many-sorted version of Theorem~\ref{thm} is now as follows, with essentially the  same proof as in the one-sorted case, {\em mutatis mutandis}:

\medskip
{\em If $\cC$ is stably embedded in $\cM$ and $\mu$ is an $A$-valued Fubini measure on $\cC$,
then $\mu$ extends uniquely to an $A$-valued Fubini measure on $\cC^{\f}$}.

\bibliographystyle{amsplain}

 \end{document}